\documentclass{amsart}

\usepackage{amsmath}
\usepackage{amsthm}
\usepackage{amsfonts}
\usepackage{amssymb}

\usepackage{xcolor}

\newcommand\R{{\mathbb{R}}}

\newcommand\Z{{\mathbf{Z}}}

\renewcommand\P{{\mathbf{P}}}

\newcommand\eps{{\varepsilon}}

\theoremstyle{plain}
  \newtheorem{theorem}{Theorem}[section]
  \newtheorem{conjecture}[theorem]{Conjecture}

  \newtheorem{proposition}[theorem]{Proposition}
  
  \newtheorem{lemma}[theorem]{Lemma}
  \newtheorem{corollary}[theorem]{Corollary}
   
\theoremstyle{remark}

\theoremstyle{definition}
  \newtheorem{definition}[theorem]{Definition}

\include{psfig}

\begin{document}

\title[Simple spectrum]{Random  matrices have simple spectrum}

\author{Terence Tao}
\address{Department of Mathematics, UCLA, Los Angeles CA 90095-1555}
\email{tao@math.ucla.edu}
\thanks{T. Tao is supported by a Simons Investigator grant, the
James and Carol Collins Chair, the Mathematical Analysis \&
Application Research Fund Endowment, and by NSF grant DMS-1266164.}

\author{Van Vu}
\address{Department of Mathematics, Yale University, New Haven 06520}
\email{van.vu@yale.edu}
\thanks{V. Vu is supported by   NSF  grant DMS 1307797  and AFORS grant FA9550-12-1-0083.}

\begin{abstract}  Let $M_n = (\xi_{ij})_{1 \leq i,j \leq n}$ be a real symmetric random matrix in which the upper-triangular entries $\xi_{ij}, i<j$ and diagonal entries $\xi_{ii}$ are independent.  We show that  with probability tending to 1, $M_n$ has no  repeated eigenvalues. As a corollary, we deduce that 
 the  Erd{\H o}s-Renyi random graph has simple spectrum asymptotically almost surely, answering a question of Babai. 
\end{abstract}

\maketitle

\setcounter{tocdepth}{2}

\section{Introduction}

Let $n$ be an asymptotic parameter going to infinity; we allow all mathematical objects in the discussion below to depend on $n$ unless explicitly declared to be fixed.  Asymptotic notation such as $o(1)$, $O()$, or $\ll$ will always be understood to be with respect to the asymptotic limit $n \to \infty$, for instance $X \ll Y$ denotes the claim that $X \leq CY$ for sufficiently large $n$ and for a fixed $C$ independent of $n$.  

In this paper, we study  the spectrum of the following general random matrix model.

\begin{definition}[A general model]\label{model}  We consider real symmetric random matrices $M_n$ of the form $M_n = (\xi_{ij})_{1 \leq i,j \leq n}$, where the entries $\xi_{ij}$ for $i \leq j$ are jointly independent with $\xi_{ji} = \xi_{ij}$, the upper-triangular entries $\xi_{ij}$, $i<j$ have distribution $\xi$ for some real random variable $\xi$ (which may depend on $n$).  The diagonal entries $\xi_{ii}$, $1 \leq i \leq n$ can have an arbitrary real distribution (and can be correlated with each other), but are required to be independent of the upper diagonal entries $\xi_{ij}$, $1 \leq i < j \leq n$.  
\end{definition}
\vskip2mm 

Important classes of matrices covered by this model include real symmetric Wigner matrix ensembles (e.g. random symmetric sign matrices) and the adjacency matrix of an Erd\"os-R\'enyi random graph $G(n,p)$ (note that we permit the diagonal entries to be identically zero).  Notice that we do not  require the distribution $\xi$ to have zero mean, or even to be absolutely integrable. As a matter of fact, our proofs do not require the entries to be iid, either; see Section \ref{section:remarks} for details. However, for sake of simplicity, we make the iid assumption in the main sections of this paper.

\vskip2mm 

The spectrum of a symmetric matrix is real, and we say that it is simple if all eigenvalues have multiplicity one.  This paper deals with the following basic question

\vskip2mm 
\centerline {\it Is it true that the spectrum of a random matrix is simple  with high probability ?} 

\vskip2mm It is easy to see that if the distribution $\xi$ is continuous, then the spectrum is simple with probability $1$. On the other hand, the discrete case is far from 
trivial.  In particular,  the following conjecture of 
Babai  has been open since the 1980s  \cite{BBcon}. 

\begin{conjecture}  \label{conjecture:Babai} 
With probability $1-o(1)$, $G(n, 1/2)$ has a simple spectrum. 
\end{conjecture}


In \cite{Babai1}, Babai, Grigoriev and Mount  showed that the notorious  graph isomorphism problem is in $P$ within the class of graphs with simple spectrum. 
Conjecture \ref{conjecture:Babai}, if holds, implies that most graphs belong to this class. 
From universality results \cite{ekyy} on the gap between adjacent eigenvalues, one can show that with probability $1-o(1)$, that \emph{most} (i.e. $(1-o(1))n$) of the eigenvalues of $G(n,1/2)$ are simple; however, the error terms in these universality results do not appear to be strong enough to resolve Babai's conjecture completely.

\vskip2mm

 Our main result provides a positive answer to the question above. We say that a real-valued random variable $\xi$ is \emph{non-trivial} if there is a fixed $\mu >0$ (independent of $n$) such that 
\begin{equation}\label{xixm}
\P( \xi =x) \le 1 -\mu
\end{equation}
 for all $x  \in \R$.  In particular, any random variable independent of $n$ is non-trivial if it is not deterministic (i.e. it does not take a single value almost surely).  The distribution $\xi$ for the adjacency matrix of $G(n,p)$ will be non-trivial if $p$ stays bounded away from both $0$ and $1$ (and in particular if $0 <p < 1$ is a fixed value such as $1/2$).

\begin{theorem}[Simple Spectrum] \label{main-thm}   Let $M_n$ be a random matrix of the form in Definition \ref{model} whose upper triangular entries have non-trivial 
distribution for some fixed $\mu>0$.  Then for every fixed $A > 0$ and  $n$ sufficiently large (depending on $A,\mu$),  the spectrum of  $M_n$ is simple with probability  at least   $1- n^{-A}$. 
\end{theorem}

In the case when $M_n$ is the adjacency matrix of  $G(n,1/2)$, $\mu =1/2$ and Theorem \ref{main-thm} implies 

\begin{corollary} 
Conjecture \ref{conjecture:Babai} holds. 
\end{corollary}


The rest of the paper is devoted to the proof of Theorem \ref{main-thm}.  One can easily extend Theorem \ref{main-thm}  (with the same proof)  to more general models where the entries are independent but not iid and also to random Hermitian matrices; see Section \ref{section:remarks}. 

\section{ From multiple eigenvalues to structured eigenvectors} 

The first step in our proof is to reduce the non-simple spectrum problem  to a 
problem about the structure of eigenvectors. 

For a symmetric matrix $M_n$ (either deterministic or random) of size $n$, write
\begin{equation}\label{mform}
M_n = \begin{pmatrix} M_{n-1} & X \\ X^* & \xi_{nn} \end{pmatrix}
\end{equation}
where $X =(x_1, \dots, x_{n-1} )   \in \R^{n-1}$ is the column vector.   We need the following (deterministic) lemma:

\begin{lemma}  Let $M_n$ be a real symmetric matrix of the form \eqref{mform}.
Assume that the spectrum of $M_n$ is not simple. Then $X$ is orthogonal to an eigenvector of $M_{n-1}$. \end{lemma}

\begin{proof} 
We can  change basis in $\R^{n-1}$ so that the standard basis $e_1,\ldots,e_{n-1}$ is an orthonormal eigenbasis of $M_{n-1}$;  $M_{n-1}$ is now a diagonal matrix with entries
$\lambda_1, \dots, \lambda_{n-1}$.  If $M_n$ has  a multiple eigenvalue, then one of the $\lambda_i$ 
is an eigenvalue of $M_n$. Assume (without loss of generality) that it is $\lambda_1$ and let $v = (v_1, \dots, v_n)$ be a corresponding eigenvector. The first row of the equation 
$M_n v = \lambda_1 v$ implies that $x_1 v_n =0$.  If $x_1=0$ then $X$ is orthogonal to $e_1$. If $v_n=0$, then $v' := (v_1, \dots, v_{n-1})$ is an eigenvector of $M_{n-1}$ and 
the last row of the equation  $M_n v = \lambda_1 v$ implies that $v' \cdot X  = \lambda_1 v_n =0$, proving the lemma. 
\end{proof}

In view of this lemma, Theorem \ref{main-thm} clearly follows from the following statement. 

\begin{proposition}\label{prop-1}  Let the notation and hypotheses be as in Theorem \ref{main-thm}, and expand $M_n$ as \eqref{mform}. 
Let $E_1$ be the event that $X$ is orthogonal to a non-trivial eigenvector of $M_{n-1}$.  Then $\P(E_1) \ll n^{-A}$.
\end{proposition}

It remains to prove this proposition.
The crucial  point here is that $X$ and $M_{n-1}$ are independent of each other.  Fix a constant  $A>0$ and call a vector $v \in \R^{n}$ \emph{rich} if we have
$$ \sup_{x \in \R} \P( X \cdot v = x) \geq n^{-A}$$  where  $X \in \R^{n}$ is a random vector whose entries are iid copies of $\xi$.   Let $E_{2,n-1}$ be the event that an eigenvector of $M_{n-1}$ is rich. 
We have 

$$\P (E_1)  = \P (E_1| \bar E_{2,n-1}) \P (\bar E_{2,n-1})  + \P( E_1| E_{2,n-1}) \P( E_{2,n-1})  \le n^{-A}+ \P( E_{2,n-1}) , $$ 
by the definition of $E_1$ and $E_{2,n-1}$.  To prove Proposition \ref{prop-1}, it therefore suffices to prove 

\begin{proposition}[Rich eigenvectors are rare]  \label{prop-2}  Let the notation and hypotheses be as above.  Then $\P(E_{2,n-1})  \ll n^{-A}$.  
\end{proposition}

In fact we will prove the stronger claim
\begin{equation}\label{stronger}
 P(E_{2,n}) \ll \exp( - c n )
\end{equation}
for some fixed $c>0$ (depending on $\sigma$), where $E_{2,n}$ is the event that an eigenvector of $M_n$ is rich; Proposition \ref{prop-2} then follows by replacing $n$ with $n-1$.

The key ingredient  of our proof of Proposition \ref{prop-2} is the so-called  \emph{inverse Littlewood-Offord theory}, introduced in \cite{TVinverse} (see \cite{NVsur} for a survey) 
which established almost completely  the  structure of rich vectors.  On the other hand, one expects that eigenvectors of a random matrix must look random, and 
thus should not attain any rigid structure. This explains the intuition behind  Proposition \ref{prop-2}.  The actual proof, 
 however, requires some novel ideas and  delicate 
arguments, and will be the subject of the next two sections.

\section{Inverse Littlewood-Offord theory}

We recall the definition of a (symmetric) \emph{generalized arithmetic progression} (GAP):

\begin{definition}
A set $P \subset \R$  is a \emph{symmetric GAP of rank $r$} if it can be expressed  in the form
$$P= \{m_1g_1 + \dots +m_r g_r: -M_i \le m_i \le M_i, m_i\in \Z \hbox{ for all } 1 \leq i \leq r\}$$
 for some $r \geq 0$, $g_1,\ldots,g_r\in \R$  and some real numbers $M_1,\ldots,M_r$. 
\end{definition} 

It is convenient to think of $P$ as the image of an integer box $B:= \{(m_1, \dots, m_r) \in \Z^r: - M_i \le m_i \le M_i \} $ under the linear map
$$\Phi: (m_1,\dots, m_d) \mapsto m_1g_1 + \dots + m_d g_d. $$
The numbers $g_i$ are the \emph{generators } of $P$, the numbers $M_i$ are the \emph{dimensions} of $P$.  We refer to $r$ as the \emph{rank} of $P$.
 We say that $P$ is \emph{proper} if this map is one to one.
We define by $\prod_{i=1}^r (2\lfloor M_r\rfloor+1) $ the volume of $P$. 
For more discussion about GAPs (including non-symmetric GAPs, which we will not use here), see \cite{TVbook}.

For a vector $V =(v_1, \dots, v_n)$, define  the \emph{concentration probability}
$$p_{\xi}(V) := \sup _{ x \in \R} \P\left( \sum_{i=1}^n \xi_i v_i = x \right) , $$ 
where $\xi_i$ are iid copies of $\xi$.  In the notation of the previous section, a vector $V$ is then rich precisely when
$p_{\xi}(V) \ge n^{-A}$.  Abusing the notation slightly, we also think of $V$ as a (multi)-set, as the ordering of the coordinates plays no role. 

The next theorem determines the structure of rich vectors/sets, asserting that such vectors mostly lie inside a GAP $P$ of bounded rank. 

\begin{theorem} [Structure theorem for rich vectors] \label{theorem:inverse}
Let $\delta <1$ and $A$ be positive constants.  There are constant $d_0 = d_0(\delta,A) \ge 1 $ and $C_0 = C_0(\delta,A)$ such that the following holds. Assume that
$$p_{\xi}(V)  \ge  n^{-A}. $$ Then for any $n^{\delta}  \le m \le n$, there exists a proper 
symmetric GAP $Q$ of rank  $r \le r_0 $  with volume at most $ C_0 p_{\xi} (v_1, \dots, v_n)^{-1} m ^{-r/2} $  such that $P$ contains all but at most $m$ elements of $V$ 
(counting multiplicities). 
\end{theorem}

\begin{proof} See \cite[Theorem 2.1]{NVoptimal}.  This theorem extended earlier results in \cite{TVinverse, TVinverse1}; see \cite{NVsur} for a survey. 
\end{proof}

For our purpose, we are going to need the following refinement of Theorem \ref{theorem:inverse}, in which the GAP $P$ not only contains most of the rich vector $V$, but also contains a large subset of $V$ that does not concentrate too strongly.
  
\begin{theorem}[A finer structure theorem for rich vectors]\label{struct-rich}  Let $0 < \eps < 1/4$ and $A > 0$ be fixed, and let $V = (v_1,\dots,v_n)$ be a rich (multi-) set.  Set $d_0 = d_0(1/2,A)$ and $C_0 = C_0(1/2,A)$ from Theorem \ref{theorem:inverse}.
Then there  are (multi-)sets  $W'\subset  W \subset V$ where   $|W|  \geq n - n^{1-\eps/4}, |W'| \le \eps n$,  a parameter $p \ge n^{-A}$, 
 and  a   GAP $P$ of  rank $d \leq d_0$ and volume at most $2C_0  p^{-1} n^{-d/2}$ such that  the following holds: 
\begin{itemize} 
\item $W \subset P$. 
\item  $p_{\xi} (W') \le n^{d_0 \eps} p $.  \end{itemize} 
\end{theorem}
 
We now prove this theorem.   We first establish the following proposition:

\begin{proposition}  \label{prop1}  Let $A \geq 0$ be fixed, and set $d_0 := d_0(1/2,A)$ and $C_0 := C_0(1/2,A)$.  Let   $P$ be a proper GAP of  rank $d \le d_0$ and volume at most $2C_0 p^{-1} n^{-d/2}$ for some 
$p \ge n^{-A}$.   Let $v_1,\ldots,v_n \in P$ (allowing repetitions).  Assume $n$ sufficiently large depending on $A$.  Then  one of the following statements holds:

\begin{itemize}
\item[(i)] (Stability) There are indices $1 \leq i_1 < \ldots < i_k \leq n$ with $k \leq \eps n$ such that the (multi-)set $V':= \{v_{i_1}, \dots, v_{i_k} \} $ satisfies 
$$ p_\xi(V')  \leq n^{d_0 \eps} p.$$

\item[(ii)] (Concentration)  There exists a 
 GAP $P'$ of some rank $d' \leq d_0 $ and volume at most $(n^{\eps/2} p)^{-1}  n^{-d'/2}$ which contains at least $n-n^{1-\eps/3}$ elements of $\{v_1,\ldots,v_n\}$.
\end{itemize}
\end{proposition}

\begin{proof} Assume that the stability conclusion (i) fails. 
 Let $I:= \{i_1,\ldots,i_k\}$ be a  subset of $\{1,\ldots,n\}$ of cardinality $k := \lfloor \eps n\rfloor$ to be chosen later, and set $V_I := \{v_{i_1} , \dots, v_{i_k} \}$. 
 Then
$$ p_\xi(V_I) > n^{d_0 \eps} p \ge n^{-A } .$$

Applying Theorem \ref{theorem:inverse} with $m:= n^{1 -\eps/2}$ and $\delta=1/2$ (and $n$ replaced by $k$),  we obtain a proper symmetric GAP $P_I$ of some rank $d_I \leq d_0$ and volume at most 

$$C_0 n^{-d_0 \eps} p^{-1} k^{-(1 -\eps/2) d'/2 }  \le  (n^{\eps/2} p)^{-1}  n^{-d'/2} $$
\noindent  which contains at least $k - n^{1-\eps/2}$  elements of $V_I$. 

At present, $P_I$ is unrelated to $P$, but we can make $P_I$ ``commensurate'' with $P$ as follows.  Write
$$ P_I = \{ n_1 w_1 + \ldots + n_{d_I} w_{d_I}: |n_i| \leq N_i \hbox{ for all } 1 \leq i \leq d_I \}$$
and let $\Sigma \subset \Z^{d_I}$ be the set of $d_I$-tuples $(n_1,\ldots,n_{d_I} ) \in \Z^{d_I}$ with $|n_i| \leq N_i$ for all $1 \leq i \leq d_I$ with $n_1 w_1 + \ldots + n_{d_I} w_{d_I} \in P$.  We say that $P_I$ has \emph{full rank} in $P$ if $\Sigma$ spans $\R^{d_I}$ as a real vector space.  We claim that we may assume without loss of generality that $P_I$ is of full rank in $P$.  Indeed, if this is not the case, then $\Sigma$ is contained in a hyperplane, which by symmetry we may take to be given by some equation $x_{d_I}  = a_1 x_1 + \ldots + a_{d_I-1} x_{d_I-1}$.  But then every element
$n_1 w_1 + \ldots + n_{d_I} w_{d_I}$ in $P_I  \cap P$ can be rewritten as
$ \sum_{j=1}^{d_I- 1} n_j (w_j + a_j w_{d_I})$
and so one may replace $P_I$ with a rank $d_I-1$ GAP $P'_I $ of volume at most that of $P_I$, such that $P_I \cap P = P'_I  \cap P$.   By the principle of infinite descent, we may iterate this procedure until we replace $P_I$ with a functionally equivalent GAP which is of full rank in $P$.

The purpose of making this full rank reduction is that it cuts down on the number of possible $P_I$.  Indeed, to specify $P_I$, one needs to specify the rank $d_I$, the dimensions $N_1,\ldots,N_{d_I}$, and the generators $w_1,\ldots,w_{d_I}$.  As $P_I$ has rank at most $d_0$  and volume at most $n^{O(1)}$ we have $O(n^{O(1)})$ choices for $d_I, N_1,\ldots,N_{d_I}$.  To specify the generators, it suffices to choose $d_I$ linearly independent elements $(n_1,\ldots,n_{d_I})$ of $\Sigma$, and their representatives $n_1 w_1 + \ldots + n_{d_I} w_{d_I}$ in $P$.  As $P$ has volume $O(n^{O(1)})$, we see that the total number of choices here is also $O(n^{O(1)})$.  Thus we see that there are at most $O(n^{O(1))}$ choices for $P_I$.

Applying the pigeonhole principle, we conclude that there exists a fixed GAP $P'$ of  rank $d' \le d_0$ and volume at most $ (n^{\eps} p)^{-1} n^{-d'/2}$ such that, when $I$  is chosen uniformly at randomly from all subsets of size $k:=\lfloor \eps n\rfloor$ of  $\{1,\ldots,n\}$ 
then with probability $\gg n^{-O(1)}$, at least $k - n^{1-\eps/2}$ of the elements of $P_I$ lie in $P'$.  A routine application of the Chernoff inequality then shows that at least $n-n^{1-\eps/3}$ of the $v_1,\ldots,v_n$ lie in $P'$.  This gives the desired concentration conclusion (ii).
\end{proof}

To prove Theorem \ref{struct-rich}, we apply Proposition \ref{prop1}  iteratively, as follows.

\begin{itemize}
\item[(i)]  Let $V$ be a rich vector, and set $p_1 := p_\xi(V)$, thus $p_1 \geq n^{-A}$ by hypothesis.  By Theorem \ref{theorem:inverse} (with $m=n^{1-\eps/2})$, we may find a GAP $P_1$ of rank $d_1 \leq d_0$ and volume at most $C_0 p^{-1} n^{-d_1/2}$ which contains all but at most $n^{1-\eps/2}$ elements of $V$.  Set $V_1 := V \cap P_1$ and $n_1 := |V_1|$, thus $n_1 \geq n - n^{1-\eps/2}$.  Initialize $i=1$, thus $V_i, P_i, n_i, p_i$ have all been defined.
\item[(ii)] If there exist a subset $V'_i$ of $V_i$ of size $k_i := \lfloor \eps n_i \rfloor$ such that $p_\xi(V'_i) \leq n_i^{d_0 \eps} p_i$, then we set $W',W,p,P$ equal to $V'_i, V_i, p_i, P_i$ respectively, and {\tt STOP}.  Otherwise, if no such subset $V'_i$ exists, we move on to step (iii).
\item[(iii)] If $p_i < n_i^{-A}$ then {\tt STOP}.  Otherwise, by Proposition \ref{prop1} with $n$ replaced by $n_i$, we may find a set $V_{i+1} \subset V_i$ of size $n_{i+1} := |V_{i+1}| \geq n_i - n_i^{1-\eps/3}$ contained in a GAP $P_{i+1}$ of rank $d_{i+1}$ at most $d_0$ and volume at most $p_{i+1}^{-1} n_{i+1}^{-d_{i+1}/2}$, where $p_{i+1} := n_i^{\eps/2} p_i$.  Thus $V_{i+1}, P_{i+1}, n_{i+1}, p_{i+1}$ have all been defined. 
\item[(iv)] Increment $i$ to $i+1$ and return to step (ii).
\end{itemize}

Note that after each successfully completed loop, the probability $p_i$ increases by a multiplicative factor of $n_i^{\eps/2}$, while $n_i$ only decreases by an additive factor of $n_i^{1-\eps/3}$; since $p_1$ is initially at least $n^{-A}$, we see that after $O(1)$ steps $p_i$ will exceed $1$, at which point we must terminate at step (ii).  Thus the above algorithm can only run for at most $O(1)$ steps.  The same analysis also shows inductively that $p_i \geq n_i^{-A}$ for all $i=O(1)$, so one never terminates at step (iii), and so must instead terminate at step (ii).  It is then routine that $W',W,p,P$ obey all the properties required for Theorem \ref{struct-rich}.

\section {Proof of Proposition \ref{prop-2} } 

We can now conclude the proof of Proposition \ref{prop-2}, and more precisely the stronger bound \eqref{stronger}.   One can show that for any given rich vector $v$, the probability that $v$ is an eigenvector is very small. Ideally, we would like to conclude 
by bounding the number of rich vectors, using the inverse theorems, and then apply the union bound (this will explain the entropy estimates below). However, this strategy does not work straightforwardly, and 
we will need to introduce an additional twist to see it through.  In this section, all implied constants in the $O()$ notation can depend on the fixed quantity $A$ (but not on the quantity $\eps$ to be introduced shortly). 

Suppose that we are in the event $E_{2,n}$, thus $M_n$ has an eigenvector $V = (v_1,\dots,v_n)$ which is rich.  Let $\eps>0$ be a small fixed quantity (depending on $A$) to be chosen later.  Applying Theorem \ref{struct-rich}, we can find (multi-)sets  $W'\subset  W \subset V$ where   $|W|  \geq n - n^{1-\eps/4}, |W'| \le \eps n$, as well as a parameter $p \gg n^{-A}$,  and  a   GAP $P$ of  rank $d = O(1)$ and volume at most $O( p^{-1} n^{-d/2} )$ such that $W \subset P$ and
$$ p_{\xi} (W') \ll n^{d_0 \eps} p.$$
By rounding $p$ to the nearest multiple of $n^{-A}$, we may assume that $p$ is an integer multiple of $n^{-A}$, which must then be of size at most $O(n^{O(d/2)})$ since otherwise $P$ would have volume less than $1$.

Write $W=\{v_{i_1}, \dots, v_{i_{n'}}  \}$ and $W' =\{v_{j_1} , \dots, v_{j_k} \}$, where $k \leq \eps n$ and $n-n^{1-\eps/4} \leq n' \leq n$.  From Stirling's formula we observe that the total number of possibilities for $p$, $d$, $n'$, $k$, $i_1,\ldots,i_{n'}$ and $j_1,\ldots,j_k$ is at most $\exp( O( n \eps \log \frac{1}{\eps} ) )$.  Thus, if we let $E_{2,n,p,d,n',k,i_1,\dots,i_{n'},j_1,\dots,j_k}$ denote the event that $M_n$ has a rich eigenvector obeying the above assertions, then by the union bound we will obtain \eqref{stronger} if we can show that
$$ \P( E_{2,n,p,d,n',k,i_1,\dots,i_{n'},j_1,\dots,j_k} ) \ll \exp( - cn )$$
for some fixed $c>0$ independent of $\eps$, and for sufficiently small $\eps$.

Let us work now with a single choice of $n,p,d,n',k,i_1,\dots,i_{n'},j_1,\dots,j_k$, and abbreviate $E_{2,n,p,d,n',k,i_1,\dots,i_{n'},j_1,\dots,j_k}$ as $E_3$, thus our task is to show that
\begin{equation}\label{pec}
\P(E_3) \ll \exp(-cn).
\end{equation}
By symmetry we may assume that $i_l = l$ for $l=1,\dots,n'$, and that $j_l = l$ for $l=1,\dots,k$.  Thus on the event $E_3$, we now have
\begin{equation}\label{vaw}
v_1,\ldots,v_{n'} \in P
\end{equation}
and
\begin{equation}\label{pavel}
 p_\xi(v_1,\dots,v_k) \ll n^{O(\eps)}  p.
\end{equation}

We cover $E_3$ by the events $E'_3$ and $E''_3$, where $E'_3$ is the event that we can take $d=0$, and $E''_3$ is the event that we can take $d>0$.

{\it Case 1:}  $d=0$.  In this case, $P$ is trivial and so $v_1=\ldots=v_{n'}=0$.  We now use a conditioning argument of Koml\'os \cite{Kom}.  Split
\begin{equation}\label{man}
M_n = \begin{pmatrix}
M_{n'} & B \\
B^* & C
\end{pmatrix}
\end{equation}
where $M_{n'}$ is the top left $n' \times n'$ minor of $M_n$, $B$ is the $n' \times n-n'$ top right minor, $B^*$ is the adjoint of $B$, and $C$ is the bottom right $n-n' \times n-n'$ minor. By hypothesis,  $M_n$ has an eigenvector $V$ with the first $n'$ coefficients vanishing, which implies from the eigenvector equation $M_n V = \lambda V$ and \eqref{man} that the matrix $B$ does not have full rank.  Thus there exists $n-n'$ rows of $B$ which span a proper subspace $H$ of $\R^{n-n'}$ in which the remaining $n' - (n-n')$ rows necessarily lie.  The entropy cost of picking these $n-n'$ rows is $\binom{n'}{n-n'}$.  Now suppose we fix the position of these rows, as well as the precise values that the random matrix attains on these rows, so that $H$ is now deterministic.  The entries of $B$ on the remaining rows remain identically distributed with law $\xi$.  This can be seen by embedding $H$ in a hyperplane, which can be written as a graph of one of the $n$ coordinates of $\R^n$ as a linear combination of the other $n-1$ coordinates, and then using \eqref{xixm}, we conclude that each of these rows has an independent probability of at most $1-\mu$ of lying in $H$.  Putting all this together, we conclude that
$$
\P(E'_3) \leq \binom{n'}{n-n'} \times (1-\mu)^{n'-(n-n')}  
$$
and hence from Stirling's formula and the size $n' = n - O(n^{1-\eps/4})$ of $n$ that
\begin{equation}\label{pe1}
\P(E'_3) \ll \exp(-cn)
\end{equation}
\noindent for some fixed $c>0$ independent of $\eps$ (but depending on $\mu$).


{\it Case 2.}  $d \geq 1$.  As in Case 1, we split $M_n$ using \eqref{man}.  Write $V':= (v_1,\ldots,v_{n'})$ (viewed as a column vector), then from the eigenvalue equation $M_n V = \lambda V$ and \eqref{man},  we see that $M_{n'} V'$ lies in the space spanned by $V'$ and the $n-n'$ columns of $B$. 
We expand the GAP $P$ as
$$P = \{ n_1 w_1+ \dots + n_d w_d: |n_i| \le N_i \}$$
for some $w_1,\dots,w_d \in \R$ and $N_1,\dots,N_d \geq 1$. By \eqref{vaw}, we have
\begin{equation}\label{vad}
V'  = w_1 V'_{(1)} + \ldots + w_d V'_{(d)}
\end{equation}
where each $V'_{(i)} \in \R^{n'}$ is a vector whose entries all lie in $[-N_i,N_i] \cap \Z$. In particular, if we let $H$ be the subspace of $\R^{n'}$ spanned by $V'_{(1)},\dots,V'_{(d)}$ and the columns of $B$, then $V'$ lies in $H$, and $H$ has dimension at most $d + (n-n') = O(n^{1-\eps/4})$.

The total number of possibilities for each vector $V'_{(i)}$ is at most $(2N_i+1)^n$. By Theorem \ref{struct-rich},  $P$ has volume
at most $O(p^{-1} n^{-d/2})$, so the total number of possibilities for $V'_{(1)},\ldots,V'_{(d)}$ is (very crudely) at most $O( p^{-1} n^{-d/2} )^n$.  Thus, by paying this as an entropy cost, we may assume that $V'_{(1)},\ldots,V'_{(d)}$ are fixed.
 For the rest of the argument, we condition on the minors $B, C$ in \eqref{man}, so the subspace $H$ defined previously is now deterministic, while the matrix $M_{n'}$ remains random (and is of the form in Definition \ref{model}, with $n$ replaced by $n'$).  The real numbers $w_1,\dots,w_d$ are also random and may potentially depend on $M_{n'}$.

We now split $M_{n'}$ further as
\begin{equation}\label{jay}
 M_{n'} = \begin{pmatrix}
M_{k} & D \\
D^* & E
\end{pmatrix}
\end{equation}
where $M_k$ is the top left $k \times k$ minor of $M_{n'}$ (or of $M_n$), $D$ is a $k \times (n'-k)$ matrix, and $E$ is a $(n'-k) \times (n'-k)$ matrix. We aim to bound the probability 
\begin{equation}\label{pq}
\P( M_{n'} V' \in H )
\end{equation}
(conditioning on $B,C$ as mentioned above); any bound we obtain on this probability, multiplied by the previously mentioned entropy cost of $O( p^{-1} n^{-d/2} )^n$, will provide a bound on $\P( E''_3 )$ by Fubini's theorem and the union bound.

To illustrate our ideas, let us first consider the toy case when $H =\{0 \}$ and the generators $w_i, 1 \le i \le d$ are fixed, which would make $V'$ deterministic.  We split
\begin{equation}\label{bone}
V'= \begin{pmatrix}
V'' \\
V'''
\end{pmatrix}
\end{equation}
where $V''$ is the column vector $(v_1,\dots,v_k)$, and $V'''$ is the column vector $(v_{k+1},\dots,v_{n'})$.  Expanding the condition $M_{n'} V' \in H$ using \eqref{jay} and \eqref{bone} and extracting the lower $n'-k$ entries of $M_{n'} V'$, we see that
$$ D^* V'' = w$$
where $w \in \R^{n'-k}$ is the vector $w := - E V'''$.  If we condition $M_k$ and $E$ to be deterministic, then $w$ becomes deterministic also, while each entry of $D^*$ remains independent with distribution $\xi$, and by \eqref{pavel} each entry of $D^* V''$ will match its corresponding entry of $w$ with an independent probability $O(n^{O(\eps)} p)$.  Multiplying these probabilities and integrating out the conditioning, we obtain the bound
$$
\P( M_{n'} V' \in H ) \ll O(n^{O(\eps)} p)^{n'-k}
$$
for \eqref{pq}, which as mentioned earlier would give an upper bound for $P(E''_3)$ of the form
$$
O( p^{-1} n^{-d/2} )^n \times O(n^{O(\eps)} p)^{n'-k}$$
which simplifies to
$$ n^{-\frac{d}{2} n + O(\eps n)} (1/p)^{k + (n-n')};$$
since $k + (n-n') = O(\eps n)$ and $p \geq n^{-A}$, this can be bounded by $\exp(-cn)$ for some fixed $c>0$ independent of $\eps$, with plenty of room to spare, if $\eps$ is chosen small enough.

There are two problems with this toy argument. Firstly, in general $H$ is not $\{0\}$. However, we will be able to (morally) reduce to the $H = \{0\}$ case by a linear projection argument that incurs a tolerable extra entropy loss (using the fact that  the dimension of $H$ is only  $O(n^{1-\eps})$). The second problem is that the number of choices for the generators $w_i$  is potentially infinite or even uncountable, so the entropy loss here is unacceptable.  We are going to avoid this problem by not counting the number of $w_i$, but a finite set of representatives, 
 which is linear algebraically equivalent.

We turn to the details.  We split
$$
V'_{(i)} = \begin{pmatrix}
V''_{(i)} \\
V'''_{(i)}
\end{pmatrix}
$$
for $i=1,\dots,d$, where $V''_{(i)} \in \R^k$ and $V''', V'''_{(i)} \in \R^{n'-k}$.    Expanding the condition $M_{n'} V' \in H$ using \eqref{jay}, \eqref{vad} and extracting the bottom $n-k$ coefficients, we conclude that
\begin{equation}\label{wavy}
w_1 (D^* V''_{(1)} + E V'''_{(1)}) + \ldots + w_d (D^* V''_{(d)} + E V'''_{(d)}) \in H_1
\end{equation}
where $H_1 \subset \R^{n'-k}$ is the projection of $H$ to $\R^{n'-k}$.  With our current conditioning, $H_1$ is a deterministic subspace of $\R^{n'-k}$ of some dimension $d_1 = O(n^{1-\eps/4})$.  Meanwhile, from \eqref{pavel}, and \eqref{vad} we have
\begin{equation}\label{pavel-2}
 p_\xi( w_1 V''_{1} + \ldots + w_d V''_{d} ) \leq n^{C \eps} p.
\end{equation}
for some fixed constant $C$ (independent of $\eps$).

We next reduce the space $H_1$ to the trivial space $\{0\}$.  Recall that $d_1 = O( n^{1-\eps/4} )$ is the dimension of $H_1$.  By permuting the indices if necessary, we may assume that\footnote{This does not incur any entropy cost, as $H_1$ was already deterministic under our current conditioning.} $H_1$ is a graph over the last $d_1$ coordinates of $\R^{n'-k}$.   In other words, we may express $H_1$ as
$$ H_1 = \{ (L(Y), Y): Y \in \R^{d_1} \}$$
for some (deterministic) linear map $L: \R^{d_1} \to \R^{n'-k-d_1}$.  Equivalently, we have
$$ H_1 = \{ (X,Y) \in \R^{n'-k-d_1} \times \R^{d_1}: \tilde L( X, Y ) = 0 \}$$
where $\tilde L: \R^{n'-k-d_1} \times \R^{d_1} \to \R^{n'-k-d_1}$ is the map
$$ \tilde L(X,Y) := X - L(Y).$$
If we identify $\R^{n'-k-d_1}$ with $\R^{n'-k-d_1} \times \{0\}$, then $\tilde L$ is the identity map on $\R^{n'-k-d_1}$ and has $H_1$ as its kernel.  Applying $\tilde L$ to \eqref{wavy}, we obtain
\begin{equation}\label{wavy-2}
w_1 \tilde L(D^* V''_{1} + E V'''_{1}) + \ldots + w_d \tilde L(D^* V''_{d} + E V'''_{d}) = 0.
\end{equation}
We now condition $M_k, E$ to be fixed; the only remaining random variables are the entries of the $k \times n'-k$ matrix $D$, which are iid with distribution $\xi$.  
Let $E_4$ denote the event that \eqref{wavy-2} holds for a given choice of deterministic data ($M_k, E, B, C, V_{(i)}, d, H, d_1, H_1$); we suppress the dependence of $E_4$ on this data.  If we can obtain an upper bound on the conditional probability that $E_4$ occurs, then by multiplying this bound by the previous entropy cost of $O( p^{-1} n^{-d/2} )^n$ would give an upper bound on $\P(E''_3)$.

We still need to control $\P(E_4)$.  Now that $H$ has been eliminated, the most significant remaining difficulty is the lack of control of the quantities $w_1,\dots,w_d$, which at present are arbitrary real numbers and can thus take an uncountable number of possible values. To resolve this difficulty we again use the conditioning arguments of Koml\'os \cite{Kom}.  Given any $m \times d$ matrix $M$ for any $m$, we say that $M$ has \emph{good kernel} if the kernel $\ker(M) := \{ w \in \R^d: Mw = 0 \}$ contains a tuple $(w_1,\dots,w_d)$ obeying both \eqref{wavy-2} and \eqref{pavel-2}.  If we form the $(n'-k-d_1) \times d$ random matrix $U$ with the vectors $\tilde L(D^* V''_{1} + E V'''_{1}), \ldots, \tilde L(D^* V''_{d} + E V'''_{d})$ as columns, then clearly $E_4$ is contained in the event that $U$ has good kernel.  

Trivially, $U$ has rank at most $d$.  As a consequence, one can select $d$ rows from $U$ whose row span is the same as that of $U$, or equivalently that the corresponding $d \times d$ minor of $U$ has the same kernel as that of $U$.  The number of possible ways to select these rows is $\binom{n'-k-d_1}{d}$, which we crudely bound by $n^d$. By paying an entropy cost of $n^d$ for the purposes of bounding $\P(E_4)$, we may thus assume that these row positions are deterministic, thus there are deterministic $1 \leq l_1 < \dots < l_d \leq n'-k-d_1$, and we need to bound the event that the $d \times d$ minor $U_d$ formed by the $l_1,\dots,l_d$ rows of $U$ has a good kernel.

We now condition on the $l_1,\dots,l_d$ rows of the $n'-k \times k$ matrix $D^*$, as well as the last $d_1$ rows of the same matrix $D^*$, thus leaving at least $n'-k-d_1-d$ of the first $n'-k-d_1$ rows of $D^*$ random (with entries independently distributed with law $\xi$).  As $\tilde L$ is the identity on $\R^{n'-k-d_1}$, we see that the $l_1,\dots,l_d$ entries of $\tilde L(D^* V''_{1} + E V'''_{1}), \ldots, \tilde L(D^* V''_{d} + E V'''_{d})$ are now deterministic (they do not depend on the remaining random rows of $D^*$).  In other words, the minor $U_d$ is now deterministic.  If $U_d$ does not have a good kernel, its contribution to $\P(E_4)$ is zero.  If instead $U_d$ has a good kernel, then we may find a \emph{deterministic} choice of $w_1,\dots,w_d$ in the kernel of $U_d$ that obeys both \eqref{wavy-2} and \eqref{pavel-2}.

The rest of the calculation is similar to the toy case. Consider the $i^{\operatorname{th}}$ component of the vector equation \eqref{wavy-2} for this deterministic choice of $w_1,\dots,w_d$, 
 where $1 \leq i \leq n'-k-d'''$ is not equal to any of the $l_1,\ldots,l_d$.  We can rewrite this component as
\begin{equation}\label{wavx}
 (w_1 V''_{1} + \ldots + w_d V''_{d}) \cdot R_i = x_i
\end{equation}
where $R_i \in \R^k$ is the $i^{\operatorname{th}}$ row of $D^*$, and $x_i \in \R$ is a deterministic quantity that does not depend on the remaining random rows in $D^*$.  By \eqref{pavel-2}, for each such $i$, the equation \eqref{wavx} holds with an independent probability of at most $n^{C\eps} p$, and so the probability that \eqref{wavy-2} holds in full is at most $O(n^{C\eps} p)^{n'-k-d_1-d}$.  Taking into account the entropy cost of $n^d$ mentioned earlier, we thus have
$$ \P(E_4) \leq n^d \times O(n^{C\eps} p)^{n'-k-d_1-d} \leq  n^{O(\eps n)} p^{n'-k-d_1-d}.$$
Paying the previously mentioned entropy cost of $O( p^{-1} n^{-d/2} )^n$, we then have
$$
\P(E''_3) \leq O( p^{-1} n^{-d/2} )^n \times n^{O(\eps n)} p^{n'-k-d_1-d} \leq n^{-\frac{d}{2}n + O(\eps n)} (1/p)^{n-n'+k+d_1+d}.$$
Since $p \geq n^{-A}$, $d \geq 1$, and $n-n'+k+d_1+d = O(\eps n)$, we conclude that
$$ \P(E''_3) \ll \exp( -cn )$$
for some fixed $c>0$ independent of $\eps$ (with plenty of room to spare), if $\eps$ is small enough.  Combining this with \eqref{pe1} we obtain \eqref{pec}.  This concludes the proof of \eqref{stronger}, and Theorem \ref{main-thm} follows.
  
\section{Concluding remarks}  \label{section:remarks} 

The assumption that the upper diagonal entries $\xi_{ij}$, $1 \leq i < j \leq n$ are iid is not essential; an inspection of the argument reveals that the proof continues to work if we assume that the $\xi_{ij}$ are independent and there is a constant $\mu >0$ such that 
$\P(\xi_{ij } =x) \le 1 -\mu$ for any $1 \le i < j \le n$ and $x \in \R$.  Our main tool, Theorem \ref{theorem:inverse}, holds under this assumption; see \cite{NVoptimal}.  The argument also easily extends to Hermitian random matrix models, in which the coefficients $\xi_{ij}$ for $i < j$ are allowed to be complex, and one imposes the condition $\xi_{ji} = \overline{\xi_{ij}}$.  In other words, the above arguments can extend to show the following result:

\begin{theorem} \label{theorem:general}   For any fixed $A, \mu \ge 0$ and sufficiently large $n$ the following holds. Let $\xi_{ij}, 1 \le i < j \le n$ be independent (complex or real) random variables 
such that  $\P(\xi_{ij } =x) \le 1 -\mu$ for any $1 \le i < j \le n$ and $x \in \R$.  Let $\xi_{ii}, 1\le i \le n$ be real random variables that are independent of the $\xi_{ij}, 1 \leq i < j \leq n$.  Set $\xi_{ji} = \overline{\xi_{ij}}$ for $1 \leq i < j \leq n$.  Then  the spectrum of the matrix $(\xi_{ij})_{1 \leq i < j \leq n}$ is simple with probability  at least   $1- n^{-A}$. 
\end{theorem}

\end{document}